\documentclass[12pt]{amsart}
\usepackage{amscd,amssymb,amsmath,latexsym,graphicx,amsxtra}

\DeclareFontFamily{OT1}{pzc}{}
\DeclareFontShape{OT1}{pzc}{m}{it}{<-> s * [1.10] pzcmi7t}{}
\DeclareMathAlphabet{\mathpzc}{OT1}{pzc}{m}{it}

\newtheorem{thm}{Theorem}[section]
\newtheorem{lem}[thm]{Lemma}
\newtheorem{prop}[thm]{Proposition}
\newtheorem{cor}[thm]{Corollary}

\newtheorem{rmk}{Remark}[section]
\newtheorem{ex}[rmk]{Example}

\newcommand \SnG{\text{$\Sigma^n(G)$}}
\newcommand \OnG{\text{$\Omega^n(G)$}}
\newcommand \OnH{\text{$\Omega^n(H)$}}

\newcommand \ep{\text{$e^{\prime}$}}

\newcommand \Rm{\text{$\mathbb{R}^m$}}

\newcommand \Hom{\text{Hom$(G,\mathbb{R})$ }}
\newcommand \bdR{\text{$\partial_{\infty}\Rm$}}

\newcommand \hgs{\text{$H_{\gamma,s}$}}
\newcommand \cgs{\text{$C_{\gamma,s}$}}
\newcommand \xgs{\text{$\Gamma_{\gamma,s}$}}

\newcommand \HxK{\text{$H \times K$}}
\begin{document}
\pagestyle{plain}

\title{The Geometric Invariants of Group Extensions}
\author{Nic Koban}
\address{Department of Mathematics, University of Maine Farmington, Farmington, ME 04938, USA}
\email{nicholas.koban@maine.edu}
\author{Peter Wong}
\address{Department of Mathematics, Bates College, Lewiston, ME 04240, USA}
\email{pwong@bates.edu}
\date{\today}
\thanks{The second author was supported in part by the National Science Foundation DMS-0805968.}
\maketitle

\begin{abstract}
In this paper, we compute the $\Sigma^n(G)$ and $\Omega^n(G)$ invariants when $1 \to H \to G \to K \to 1$ is a short exact sequence of finitely generated groups with $K$ finite. We also give sufficient conditions for $G$ to have the $R_{\infty}$ property in terms of $\Omega^n(H)$ and $\Omega^n(K)$ when either $K$ is finite or the sequence splits.  As an application, we construct a group $F\rtimes_{\rho} \mathbb Z_2$ where $F$ is the R.\ Thompson's group $F$ and show that $F\rtimes_{\rho} \mathbb Z_2$ has the $R_{\infty}$ property while $F$ is not characteristic.
\end{abstract}

\section{Introduction}

The Bieri-Neumann-Strebel-Renz invariants (and their homological analogs) $\Sigma^n(G)$  of a group have been useful in obtaining finiteness properties of subgroups of $G$ with abelian quotients. Connections to other areas of mathematics have been made while the computation of these invariants remains difficult in general. In fact, the so-called direct product conjecture for $\Sigma^n (\HxK)$ has been shown to be false in general (\cite{MMV} for the homotopical version and \cite{schutz} for the homological, although in \cite{BG}, the product conjecture for the homological version of the $\Sigma$-invariants is proven over a field). On the other hand, an analogous geometric invariant $\Omega^n(G)$ has been introduced and has proven somewhat easier to compute. For example, the product formula $\Omega^n(\HxK)=\Omega^n(H) \circledast \Omega^n(K)$, the spherical join of $\Omega^n(H)$ and $\Omega^n(K)$, holds. More recently, the product formula for $\Omega^n$ has been employed to yield new families of groups for which the $R_{\infty}$ property holds \cite{KW}. The $R_{\infty}$ property arises from the study of twisted conjugacy classes of elements of the fundamental group in topological fixed point theory (see \S~\ref{tcomega}).

Motivated by \cite{KW} and the product formula for $\Omega^n$ \cite{K2}, we conjectured similar formulas for $\Omega^1$ for finite and split extensions in the unpublished manuscripts \cite{KW3} and \cite{KW4}. We have since found counterexamples to the formulas.  Although these formulas are false in general, we have found use for the formulas which we discuss in \S~\ref{tcomega} of this paper.

The main objective of this paper is to use the $\Sigma$- and $\Omega$-invariants to detect the $R_{\infty}$ property for an extension $1 \to H \to G \to K \to 1$ where either $K$ is finite or the sequence splits.  As an application, we construct a finitely presented group with $\Omega^2$ consisting of a single discrete point and hence with the $R_{\infty}$ property whereas $\Omega^1$ contains two antipodal points from which the $R_{\infty}$ property cannot be detected.  This is the first example known to the authors where $n=2$ is needed to obtain this property where $n=1$ fails to satisfy the conditions.

This paper is organized as follows.  The invariants $\Omega^n$ were defined in \cite{K1} and are analogs of the Bieri-Neumann-Strebel-Renz invariants $\Sigma^n$ defined in \cite{BNS} for $n=1$ and in \cite{BR} for $n \geq 2$.  We recall these definitions in \S~\ref{sigma} and \S~\ref{omega}.  In \S~\ref{hom}, we describe the real vector space of characters, Hom$(G,\mathbb{R})$, for a finite and split extension $G$ in terms of Hom$(H,\mathbb{R})$, Hom$(K,\mathbb{R})$, and the action of $K$ on $H$.  In \S~\ref{finite}, we prove the formula for $\Sigma^1(G)$ where $G$ is a finite extension.  In \S~\ref{omega-finite}, we investigate the conjectured $\Omega$-formula for finite extensions, and give examples where this formula fails and conditions when the formula holds.  In \S~\ref{tcomega}, we give conditions using the $\Omega$-invariant to detect the $R_{\infty}$ property in finite and split extensions.

\section{The Geometric Invariants $\Sigma$ and $\Omega$}\label{invar}
Let $G$ be a finitely generated group with generating set $S$.  In this section, we define two invariants of $G$:
	\begin{enumerate}
	\item the Bieri-Neumann-Strebel (or BNS) invariant $\Sigma$, and
	\item the invariant $\Omega$.
	\end{enumerate}
	
\subsection{The BNS invariant $\Sigma$}\label{sigma}
The set \Hom of homomorphisms from $G$ to the additive group of reals is a real vector space with dimension equal to the ${\mathbb Z}$-rank of the abelianization of $G$, so \Hom $\cong \Rm$ for some $m$.  Thus, there is a natural isomorphism between \Hom and the real vector space $G/G^{\prime} \otimes_{\mathbb Z} \mathbb{R}$.  The group $G$ acts on $G/G^{\prime} \otimes_{\mathbb Z} \mathbb{R}$ by left multiplication on the $G/G^{\prime}$ component, and this gives an action of $G$ on \Hom (by translations).

Let $\Gamma$ denote the Cayley graph of $G$ with respect to a chosen generating set.  Define $\mathpzc{h}:\Gamma \to \Rm$ to be the abelianization map on the vertices and extend linearly on the edges.  Denote by $\bdR$ the boundary at infinity of $\Rm$ (ie.\ the set of geodesic rays in $\Rm$ initiating from the origin).  This is isomorphic to the {\it character sphere of $G$} defined as the set of equivalence classes $S(G) := \{ [\chi] | \chi \in {\rm Hom}(G,\mathbb{R}) - \{ 0 \} \}$ where $\chi_1 \sim \chi_2$ if and only if $\chi_1 = r\chi_2$ for some $r > 0$.  Let $e \in \bdR$ and let $\gamma$ be the geodesic ray defining $e$.  We denote by $H_e$ the half space perpendicular to $\gamma$ that contains all of the image of $\gamma$.  Denote by $\Gamma_e$ the largest subgraph of $\Gamma$ that is contained in $\mathpzc{h}^{-1}(H_e)$.  The direction $e \in \Sigma^1(G)$ if $\Gamma_e$ is path connected.

We give an equivalent definition for $\Sigma^1(G)$ that we will use in this paper.  It will also be useful to see the motivation for the definition of $\Omega^1(G)$.  Let $e \in \bdR$ and let $\gamma$ be a geodesic ray defining $e$.  For each $s \in {\mathbb R}$, let $\hgs$ be the closed half-space orthogonal to $\gamma$ so that $\hgs \cap \gamma([0,\infty)) = \gamma([s,\infty))$.  For each $s \in {\mathbb R}$, denote by $\xgs$ the largest subgraph of $\Gamma$ contained in $\mathpzc{h}^{-1}(\hgs)$.  The direction $e \in \Sigma^1(G)$ if and only if for every $s \geq 0$, there exists $\lambda = \lambda(s) \geq 0$ such that any two points $u,v \in \Gamma_{\gamma,s}$ can be joined by a path in $\Gamma_{\gamma,s-\lambda}$ and $s-\lambda(s) \to \infty$ as $s \to \infty$.

\subsection{The invariant $\Omega$}\label{omega}
In the compactified space $\Rm \cup \bdR$, the compactified half-spaces play
the role of neighborhoods of the point $e \in \bdR$, but this
gives an unsatisfactory topology to $\Rm \cup \bdR$.
From the point of view of topology, it is more natural to have a similar definition to $\Sigma^1(G)$ using ``ordinary" neighborhoods of $e$.  A basis for these neighborhoods consists of ``truncated cones".  For each $s \geq 0$, define the {\it truncated cone} $\cgs := Cone_{\theta}(\gamma) \cap \hgs$ where $Cone_{\theta}(\gamma)$ is the closed cone of angle $\theta$ and vertex $\gamma(0)$ and $\theta := \arctan(\frac{1}{s})$ if $s > 0$ and $\theta := \frac{\pi}{2}$ if $s=0$.  For each $s \geq 0$, denote by $\Delta_{\gamma,s}$ the largest subgraph of $\Gamma$ contained in $\mathpzc{h}^{-1}(\cgs)$.  We say that $e \in \Omega^1(G)$ if and only if there exists $s_0 \geq 0$ such that for each $s \geq s_0$, there exists $\lambda = \lambda(s) \geq 0$ such that any two points $u,v \in \Delta_{\gamma,s}$ can be joined by a path in $\Delta_{\gamma,s-\lambda}$ and $s-\lambda(s) \to \infty$ as $s \to \infty$.

When $n > 1$, we can make the following changes to the definitions to obtain $\Sigma^n(G)$ and $\Omega^n(G)$: replace $G$ being finitely generated with type $F_n$, replace the Cayley graph $\Gamma$ with an $n$-dimensional, $(n-1)$-connected CW-complex $X$ on which $G$ acts freely as a group of cell permuting homeomorphisms with $G \backslash X$ a finite complex, $\mathpzc{h}$ is a $G$-map from $X$ to \Hom, and replace the path-connected property in the definition with the analogous $(n-1)$-connected property (see \cite{BR} for $\Sigma$ and \cite{K1} for $\Omega$).   We should mention that $\Omega^n(G)$ is always a closed set while $\Sigma^n(G)$ is open.

The following theorem relates the invariants $\SnG$ and $\OnG$.
\begin{thm}\label{hemi}\cite[Theorem 3.1]{K1}
Let $e \in \bdR$.  Then $e \in \OnG$ if and only if $\ep \in \SnG$ for every $\ep$ in an open $\frac{\pi}{2}$-neighborhood of $e$.
\end{thm}

Given $\Sigma^n(G)$, we can completely determine
$\Omega^n(G)$: for each $e \in \bdR$, $e \in \Omega^n(G)$ if and only if
the open $\frac{\pi}{2}$-neighborhood of $e$ is in $\Sigma^n(G)$.
However, it is not the case that $\Omega^n(G)$ completely determines
$\Sigma^n(G)$; examples of such groups are given in \cite[\S~1.3]{K1}.

The following theorem completely describes $\Omega^n(H \times K)$ in terms of $\OnH$ and $\Omega^n(K)$.  This theorem will be useful in \S~\ref{tcomega}.
\begin{thm}\label{product-formula}\cite[Theorem 3.8]{K2}
$\Omega^n(H \times K) = \OnH \circledast \Omega^n(K)$ where $\circledast$ represents the spherical join.
\end{thm}

\section{The Space Hom$(G,\mathbb{R})$}\label{hom}

Consider a group extension $G$ given by the following short exact sequence of groups
$$
1 \to H \stackrel{i}{\to} G \stackrel{p}{\to} K \to 1.
$$
Without loss of generality we may assume that the monomorphism $i$ is an inclusion so that $H$ is identified with its image in $G$ as a normal subgroup. Let $\tilde K$ be a left transversal for $H$, i.e., $\tilde K=\nu(K)$ where $\nu:K\to G$ is an injective function so that $(p\circ \nu)(k)=k$ for all $k\in K$. In addition, we assume that $\nu(1_K)=1_G$. There is an {\it ``action"} of $\tilde K$ on ${\rm Hom}(H,\mathbb R)$ given by
$$
(\tilde k \cdot \phi)(h)=\phi(\tilde k h \tilde k^{-1})
$$
for all $h\in H$ where $\phi \in {\rm Hom}(H,\mathbb R)$.

Let
$$
Fix \hat{\nu}=\{\phi \in {\rm Hom}(H,\mathbb R)|\tilde k \cdot \phi=\phi, ~\forall \tilde k \in \tilde K=\nu(K)\}.
$$
Note that $Fix \hat \nu$ is a vector subspace of ${\rm Hom}(H,\mathbb R)$.

\begin{rmk} If $\nu_1,\nu_2:K\to G$ are two left transversals for $H$ then $Fix\hat \nu_1=Fix \hat \nu_2$. To see this, let $\phi\in Fix \hat \nu_1$ and $\tilde k_2\in \nu_2(K)$. Write $\tilde k_2=\nu_2(k)$ for some $k\in K$. Since $(p\circ \nu_i)(k)=k$ for $i=1,2$, it follows that $\nu_2(k)=\nu_1(k)h'$ for some $h'\in H$. Now,
\begin{equation*}
\begin{aligned}
(\nu_2(k)\cdot \phi)(h)&=\phi(\nu_2(k)h\nu_2(k)^{-1}) \\
                       &=\phi(\nu_1(k)h'h(h')^{-1}\nu_1(k)^{-1}) \\
                       &=\phi(h'h(h')^{-1}) \quad \text{since~} \phi\in Fix \hat \nu_1 \\
                       &=\phi(h).
\end{aligned}
\end{equation*}
Thus, $\phi \in Fix \hat \nu_2$. A similar argument shows that if $\phi \in Fix \hat \nu_2$ then $\phi \in Fix \hat \nu_1$.
\end{rmk}

In the special case when $G=H\rtimes_{\rho} K$ is the semi-direct product given by an action $\rho:K\to {\rm Aut}(H)$, the canonical left transversal $\nu$ is the section given by $\nu(k)=(1,k)$. Then $Fix \hat \nu=\{\phi \in {\rm Hom}(H,\mathbb R)| \hspace{2pt} \phi(\rho(k)(h))=\phi(h)$ for all $h \in H, k \in K \}$. In this case, we also write $Fix \hat \rho$ for $Fix \hat \nu$ as the fixed subspace induced by the action $\rho$.

Following \cite{johnson}, we will use the presentations for $H$ and $K$ to derive a presentation for $G$.  Let $H \cong \langle A | R \rangle$ and $K \cong \langle B | S \rangle$.  Since every word in $S$ is equivalent to the identity, $\nu(S) \subseteq ker(p) = H$.  Thus, every word $\nu(S)$ is equivalent to some word $w_s$ in $A$.  Denote by $X := \{ \nu(s)w_s^{-1} | s \in S$ and $w_s$ is the equivalent word to $\nu(s)$ in $A \}$.  For each $\tilde{b} \in \nu(B)$ and each $a \in A$, there is a word $w_{a,b} \in H$ such that $\tilde{b}a\tilde{b}^{-1} = w_{a,b}$.  Let $Y := \{ \tilde{b}a\tilde{b}^{-1}w_{a,b}^{-1} | \tilde{b} \in \nu(B); a \in A \}$.  Then $G \cong \langle A \cup \nu(B) | R \cup X \cup Y \rangle$.

\begin{prop}\label{finitefix}
Let $G$ be a finite extension given by the short exact sequence of groups
$$
1\to H \stackrel{i}{\to} G \stackrel{p}{\to} K \to 1
$$
where $K$ is finite and let $\nu: K\to G$ be a left transversal for $H$ such that $\nu(1_K)=1_G$. Then
${\rm Hom}(G,\mathbb R) \cong Fix\hat{\nu}$.
\end{prop}

\begin{proof} Suppose $\phi \in Fix\hat{\nu}$.  Thus, $\phi$ is defined on $A$, so we need only define $\phi$ on $B$ so that it satisfies the relations $X$ and $Y$ of $G$ to get a homomorphism from $G$ to the reals.  Since $K$ is finite, for each $b \in B$, there is a relation in $S$ of the form $b^m$ for some integer $m \geq 1$.  Thus, there is a word $w$ in $A$ such that $\tilde{b}^m w^{-1} = 1$, so $\phi(\tilde{b}) = \frac{\phi(w)}{m}$.  Since $\phi \in Fix\hat{\nu}$, the relations in $Y$ are satisfied, and obviously the relations in $X$ are satisfied.  Let $\hat{\phi}$ be this extension of $\phi$.  Define $T: Fix \hat \nu \to {\rm Hom} (G)$ by $T(\phi) = \hat{\phi}$.

If $\phi \in {\rm Hom}(G,\mathbb R)$ then define $Q:{\rm Hom}(G,\mathbb R) \to Fix \hat \nu$ by $Q(\phi)=\phi\circ i$. A priori, $Q(\phi)\in {\rm Hom}(H,\mathbb R)$. To see that the image actually lies in $Fix \hat \nu$, we note that
\begin{equation*}
\begin{aligned}
\tilde k\cdot (\phi\circ i))(h)&= (\phi\circ i)(\tilde k h\tilde k^{-1}) \\
                               &=\phi(\tilde k h \tilde k^{-1}) \\
                               &=\phi(\tilde k)\phi(h)\phi(\tilde k)^{-1} \qquad [\text{since~}\phi\in {\rm Hom}(G,\mathbb R)] \\
                               &=\phi(h).
\end{aligned}
\end{equation*}
It follows that $(\phi \circ i)\in Fix \hat \nu$. It is easy to see that $Q\circ T$ and $T\circ Q$ yield identity maps and thus the assertion follows.
\end{proof}

We now give the analogous proposition in the split extension case.

\begin{prop}\label{fix}
Let $G$ be a split extension given by the short exact sequence of groups
$$
1\to H \stackrel{i}{\to} G \stackrel{p}{\to} K \to 1
$$
and let $\nu: K\to G$ be a left transversal for $K$ such that $\nu(1_K)=1_G$. Then
${\rm Hom} (G,\mathbb R) \cong Fix \hat \rho \times {\rm Hom} (K,\mathbb R).$
\end{prop}

\begin{proof}
Define $\Phi:{\rm Hom}(G,\mathbb R) \to Fix \hat \rho \times {\rm Hom}(K,\mathbb R)$ by $\phi \mapsto (\phi \circ i, \phi \circ \sigma)$.  To show that $\Phi$ is well-defined, we first show that $\phi \circ i \in Fix \hat \rho$.  Let $h \in H$ and $k \in K$, so

\begin{center}
$(k \cdot \phi \circ i)(h) = \phi \circ i(\tilde{k}h\tilde{k}^{-1}) = \phi(\tilde{k}h\tilde{k}^{-1}) = \phi(\tilde{k}) + \phi(h) - \phi(\tilde{k}) = \phi \circ i(h)$.
\end{center}

Define $\Psi: Fix \hat \rho \times {\rm Hom}(K,\mathbb R) \to {\rm Hom}(G,\mathbb R)$ by $(\alpha,\beta) \mapsto (\hat \alpha + \beta \circ \pi)$ where $\hat \alpha (g) = \hat \alpha (h\tilde{k}) := \alpha(h)$ and $\tilde{k} = \sigma \circ \pi (g)$.  Since $ker(\pi) = H$, $(\beta \circ \pi)(h)=0$ for all $h\in H$.  To show that $\hat \alpha$ is a homomorphism, we note that since $\alpha \in Fix \hat \rho$, we have $\alpha(\tilde{k}h\tilde{k}^{-1}) = \alpha(h)$.  Therefore,

\begin{center}
$\hat{\alpha}(g_1g_2) = \hat{\alpha}(h_1\tilde{k}_1 h_2\tilde{k}_2) = \hat{\alpha}(h_1(\tilde{k}_1 h_2 \tilde{k}_{1}^{-1})\tilde{k}_1\tilde{k}_2) = \alpha(h_1)\alpha(\tilde{k}_1 h_2 \tilde{k}_{1}^{-1}) = \alpha(h_1)\alpha(h_2) = \hat{\alpha}(g_1)\hat{\alpha}(g_2)$.
\end{center}

To see that $\Phi$ and $\Psi$ are inverses, we have

\begin{center}
$\Phi \circ \Psi(\alpha,\beta)(h,k) = ((\hat{\alpha} + \beta \circ \pi) \circ i, \displaystyle (\hat{\alpha} + \beta \circ \pi) \circ \sigma)(h,k) = (\hat{\alpha}(h) + \beta \circ \pi(h), \beta \circ \pi(\sigma(k))) = (\alpha(h),\beta(k))$
\end{center}

and

\begin{center}
$\Psi \circ \Phi (\phi)(g) = (\widehat{\phi \circ i} + \phi \circ \sigma \circ \pi)(g) = \widehat{\phi \circ i}(h\tilde{k}) + \phi(\sigma(\pi(h\tilde{k}))) = \phi(h) + \phi(\tilde{k}) = \phi(g)$.
\end{center}
\end{proof}

\section{The $\Sigma$-invariant for finite extensions}\label{finite}

In this section, we prove for a finite extension $1 \to H \to G \to K \to 1$ that $\Sigma^1(G)=\partial_{\infty} Fix \hat \nu \cap \Sigma^1(H)$.  We should mention that a more general result for finite index subgroups was given in \cite{MMV}.

\begin{thm}\label{index}\cite[Theorem 9.3]{MMV}
Suppose that $H \leq G$ is a subgroup of finite index, and that $\chi$ restricts to a non-zero homomorphism of $H$.  Then $[\chi|_H] \in \Sigma^n(H)$ if and only if $[\chi] \in \Sigma^n(G)$.
\end{thm}

We give a geometric proof of Theorem~\ref{main} for $n=1$.  The authors would like to thank Dessislava Kochloukova for pointing us to the result in \cite{MMV}.

\begin{prop}\label{inclusion}
Given an extension $1 \to H \to G \to K \to 1$ where $K$ is finite, if $H$ is finitely generated, then $\Sigma^1(G) \subseteq \Sigma^1(H)$.
\end{prop}
\begin{proof} Since $H$ is finitely generated and $K$ is finite, $G$ is finitely generated. First, we note that since ${\rm Hom}(G,\mathbb{R}) = Fix\hat{\nu}$, we have ${\rm Hom}(G,\mathbb R) \subseteq {\rm Hom}(H,\mathbb R)$ as a vector subspace.
Let $\Gamma_H$ and $\Gamma_G$ be the Cayley graphs of $H$ and of $G$ with respect to the generating sets $A$ and $A\cup \nu(B)$ respectively. For any height function $\mathpzc{h}:\Gamma_G \to {\rm Hom}(G,\mathbb R)$, we have an induced height function $\bar{\mathpzc{h}}:\Gamma_H \to {\rm Hom} (H,\mathbb R)$ such that the following diagram commutes:
\begin{equation*}
\begin{CD}
\Gamma_H     @>\bar{\mathpzc{h}}>>   {\rm Hom}(H,\mathbb R)\\
@V{\hat i}VV                           @AA{\hat Q}A\\
\Gamma_G         @>\mathpzc{h}>>    {\rm Hom}(G,\mathbb R)
\end{CD}
\end{equation*}
where $\hat Q$ is the induced inclusion ${\rm Hom}(G,\mathbb R) \subseteq {\rm Hom}(H,\mathbb R)$ from the proof of Proposition\ref{finitefix}, and $\hat i$ is induced by the inclusion $i: A\subseteq A \cup \tilde B$. In fact, $\hat i$ is the inclusion of $\Gamma_H$ as the subgraph induced by the vertices $\{h\nu(1_K)|h\in H\}$ in $\Gamma_G$.

Let $e\in \Sigma^1(G)$. We will show that $e\in \Sigma^1(H)$. First, take a half space $\bar H_{e,s}$ for $e$ in ${\rm Hom}(H,\mathbb R)$. Since ${\rm Hom}(G,\mathbb R) \subseteq {\rm Hom}(H,\mathbb R)$, there is a unique half space $H_{e,s}$ in ${\rm Hom}(G,\mathbb R)$ such that $H_{e,s}=\bar H_{e,s} \cap {\rm Hom}(G,\mathbb R)$. Take $x$ and $y$ to be vertices in $\bar{\mathpzc{h}}^{-1}(\bar H_{e,s})$. Since $\bar{\mathpzc{h}}=\hat Q \circ \mathpzc{h} \circ \hat i$, it follows that $\bar{\mathpzc{h}}^{-1}(\bar H_{e,s}) \subseteq \mathpzc{h}^{-1}(H_{e,s})$ and thus $x,y\in \mathpzc{h}^{-1}(H_{e,s})$, that is, there exist $g_1,g_2,..., g_m\in (A\cup \tilde B)$ such that $y=xg_1g_2...g_m$ and for each $1 \leq i \leq m$, $\mathpzc{h}(xg_1 \ldots g_i)\subset H_{e,s}$. Let $\mu := {\rm min} \{ \mathpzc{h}(\tilde{k} | \tilde{k} \in \nu(K) \}$.  We can rewrite the path from $y$ to $x$ as $h_1k_1h_2k_2 \ldots h_{p-1}k_{p-1}h_p$ where each $h_j \in H$ and each $k_j \in \nu(K)$ with the possibility that $h_1$ and $h_p$ are the identity element.  Since $H$ is normal, we have for each $1 \leq i \leq p-1$, $k_i h_{i+1} = \bar{h}_{i+1} k_i$, and $\mathpzc{h}(x\bar{h}_1 \ldots \bar{h}_{i+1}) \subset H_{e,s-\mu}$.  Therefore, $y = x\bar{h}_1 \ldots \bar{h}_p k_1 \ldots k_{p-1}$, and since $x, y, (\bar{h}_1 \ldots \bar{h}_p) \in H$, we have that $k_1 \ldots k_{p-1}$ is trivial (otherwise it would be a non-trivial element of $H$).  Thus, $y = x\bar{h}_1 \ldots \bar{h}_p$ and for each $1 \leq i \leq p$, $\mathpzc{h}(x\bar{h}_1 \ldots \bar{h}_i) \subset H_{e,s-\mu}$.
Hence $x$ and $ y$ are connected in $\bar{\mathpzc{h}}^{-1}(\bar H_{e,s})$ or $e\in \Sigma^1(H)$.
\end{proof}

\begin{rmk} Note that the inclusion ${\rm Hom}(G,\mathbb R) \subseteq {\rm Hom}(H,\mathbb R)$ of Proposition \ref{inclusion} can be strict. For instance, take $H=\mathbb Z^2$ and $G$ to be the fundamental group of the Klein bottle with $K=\mathbb Z_2$. Here, ${\rm rk}_{\mathbb Z}(G)=1 < 2={\rm rk}_{\mathbb Z}(\mathbb Z^2)$. Moreover, Proposition \ref{inclusion} is false if $K$ is not finite, e.g., take $G=\mathbb Z^2$ and $H=\mathbb Z$.
\end{rmk}

\begin{thm}\label{main}
Let $G$ be a finite extension given by the short exact sequence of groups
$$
1\to H \stackrel{i}{\to} G \stackrel{p}{\to} K \to 1
$$
where $K$ is finite, $H$ is finitely generated, and let $\nu: K\to G$ be a left transversal for $K$ such that $\nu(1_K)=1_G$.
Then, $\Sigma^1(G)=\partial_{\infty} Fix \hat \nu \cap \Sigma^1(H)$.
\end{thm}

\begin{proof}
By Prop. \ref{inclusion}, $\Sigma^1(G)\subseteq \Sigma^1(H)$. Since $K$ is finite, we have ${\rm Hom}(G,\mathbb{R})=Fix \hat \rho$ by Prop. \ref{finitefix}. Thus, $\Sigma^1(G)\subseteq \partial_{\infty}{\rm Hom}(G)=\partial_{\infty} Fix \hat \rho$, and we have $\Sigma^1(G)\subseteq \partial_{\infty} Fix \hat \rho \cap \Sigma^1(H)$.

Suppose $e\in \partial_{\infty} Fix \hat \rho \cap \Sigma^1(H)$ and suppose $\gamma$ is a geodesic ray defining $e$. To show that $e \in \Sigma^1(G)$, let $s \in \mathbb{R}$, and let $H_{\gamma,s}$ be the corresponding half space of $e$ in ${\rm Hom}(G,\mathbb{R})=Fix \hat \rho \subseteq {\rm Hom}(H,\mathbb R)$. Let $\pi: G \to G/G^{\prime} \cong \mathbb{Z}^m$ be the natural projection epimorphism.  Let $\Gamma_G$ be the Cayley graph of $G$ with respect to the generating set $A \cup \nu(B)$ from the presentations $H = \langle A | R \rangle$ and $K = \langle B | S \rangle$.  Define $\mathpzc{h}: \Gamma_G \to {\rm Hom}(G,\mathbb{R}) \cong \mathbb{R}^m$ by: $\mathpzc{h}(g) = \pi(g)$ for all vertices $g \in \Gamma_G$, and extend linearly on edges.  Choose two points $x,y \in \mathpzc{h}^{-1}(H_{\gamma,s})$.  Since $G$ is a finite extension, $x$ and $y$ can be uniquely written as $x=h_1 \tilde{k}_1$ and $y=h_2 \tilde{k}_2$ for some $h_1, h_2 \in H$ and $k_1, k_2 \in K$.
Let $\lambda_1 := {\rm min} \{ d(\mathpzc{h}(\tilde{k}), H_{\gamma,0}) | \tilde{k} \in \nu(K) \}$ where $d(\mathpzc{h}(\tilde{k}), H_{\gamma,0})$ is the distance between the point $\mathpzc{h}(\tilde{k})$ and the half space $H_{\gamma,0}$.  Since $x,y \in \mathpzc{h}^{-1}(H_{\gamma,s})$, we have $h_1, h_2 \in \mathpzc{h}^{-1}(H_{\gamma,s-\lambda_1})$.  Since $e \in \Sigma^1(H)$, there exists $\lambda_2 \geq 0$ such that there is a path $w$ in $(\Gamma_H)_{\gamma,s-\lambda_1 - \lambda_2} \subseteq (\Gamma_G)_{\gamma,s - \lambda_1 - \lambda_2}$ from $h_1$ to $h_2$. Thus, $\tilde{k}_1^{-1}w\tilde{k}_2$ is a path in $(\Gamma_G)_{\gamma-\lambda_1-\lambda_2}$ from $x$ to $y$.  Since $K$ is finite, $s - \lambda_1 - \lambda_2 \to \infty$ as $s \to \infty$.
\end{proof}

\begin{rmk}\label{ngeq1}
It should be noted that due to Theorem~\ref{index} and Proposition~\ref{finitefix} the result in Theorem~\ref{main} holds for all $n \geq 1$.  We will use this fact in later constructions in this paper.
\end{rmk}

\begin{cor} For the semi-direct product $H\rtimes_{\rho} K$ of a finitely generated group $H$ and a finite group $K$, we have
$\Sigma^1(H\rtimes_{\rho}K)=\Sigma^1(H) \cap \partial_{\infty}Fix \hat \rho$.
\end{cor}

\begin{ex}\label{klein1}
Consider the fundamental group of the Klein Bottle with the standard presentation
$$
G=\langle \alpha, \beta |\alpha \beta\alpha \beta^{-1}=1\rangle.
$$
The subgroup generated by $\alpha$ and $\beta^2$ is isomorphic to $\mathbb Z^2$ since $\alpha \beta^2=\alpha \beta(\alpha \beta\alpha)=(\alpha \beta \alpha)(\beta\alpha)=\beta(\beta(\alpha)=\beta^2\alpha$. This subgroup is the fundamental group of the $2$-torus as a double cover of the Klein Bottle. Moreover, $G$ admits the following short (non-split) exact sequence
$$
0\to \langle \alpha, \beta^2\rangle \to G \stackrel{p}{\to} \mathbb Z_2=\langle \bar{\beta}|\bar{\beta}^2=\bar 1\rangle \to 0
$$
where the projection $p$ sends $\alpha$ to $\bar 1$ and $\beta$ to $\bar{\beta}$.It follows from Theorem \ref{main} that $\Sigma^1(G)=\Sigma^1(\mathbb Z^2)\cap \partial_{\infty}Fix \hat \nu$ where $\nu:\mathbb Z_2 \to G$ is given by $\bar{\beta}\mapsto \beta$. Thus,
\begin{equation*}
\begin{aligned}
\Sigma^1(G) &=\mathbb S^1 \cap \partial_{\infty}Fix \hat \nu \\
            &=\mathbb S^1 \cap \partial_{\infty}\{\phi\in {\rm Hom}(\mathbb Z^2)|\beta \cdot \phi=\phi\} \\
            &=\mathbb S^1 \cap \partial_{\infty}\{\phi\in {\rm Hom}(\mathbb Z^2)|\phi(\beta h \beta^{-1})=\phi(h), \forall h\in H\}.
\end{aligned}
\end{equation*}
Note that $H=\langle \alpha, \beta^2\rangle$ and $\phi(\beta \alpha \beta^{-1})=\phi(\alpha)$ implies that $\phi(\alpha^{-1})=\phi(\alpha)$ which in turn implies that $\phi(\alpha)=0$. This implies that $Fix \hat \nu=\mathbb R$ and so $\Sigma^1(G)=\{\pm \infty\}$. Since ${\rm Hom}(G,\mathbb R) =Fix \hat \nu=\mathbb R$ is one dimensional, we have $\Omega^1(G)=\Sigma^1(G)=\{\pm \infty\}$.
\end{ex}

\begin{ex}\label{D-infinity}
Consider the infinite dihedral group $D_{\infty}$. It is known that $D_{\infty} \cong \mathbb Z_2 * \mathbb Z_2$, the free product of $\mathbb Z_2$ with $\mathbb Z_2$. Moreover, it is also isomorphic to $\mathbb Z \rtimes_{\rho} \mathbb Z_2$ where the action $\rho$ is the non-trivial one. It is easy to see that $Fix\hat \rho$ is the origin $\{0\}$ so that $\partial_{\infty}Fix \hat \rho=\emptyset$. It follows from Theorem \ref{main} that $\Sigma^1(D_{\infty})=\emptyset=\Omega^1(D_{\infty})$. On the other hand, $D_{\infty}$ admits the following non-split extension
$$
1\to (\mathbb Z_2 * \mathbb Z_2)' \to \mathbb Z_2 * \mathbb Z_2 \to \mathbb Z_2 \times \mathbb Z_2 \to 0
$$
where $Z'$ denotes the commutator subgroup of a group $Z$. Furthermore, $(\mathbb Z_2 * \mathbb Z_2)'$ is isomorphic to the infinite cyclic group $\mathbb Z$ (see e.g. Exercise 10 on p.134 of \cite{johnson}). Again similar arguments show that, using Theorem \ref{main}, that $\Sigma^1(D_{\infty})=\emptyset=\Omega^1(D_{\infty})$.
\end{ex}

\section{The $\Omega$-invariant for finite extensions}\label{omega-finite}

The authors originally conjectured that $\Omega^1(G) = \Omega^1(H) \cap \partial_{\infty}Fix\hat{\nu}$.  This however turned out not to be true as the following examples show.

\begin{ex}\label{ex1}
Recall that the R. Thompson's group $F$ can be given the following presentation
$$
F=\langle x_0,x_1,x_2,...| x_{k}^{-1}x_nx_k=x_{n+1}, k<n\rangle.
$$
The elements $x_0$ and $x_1$ correspond to the following piecewise linear homeomorphisms of the unit interval:
\begin{equation}\label{x0}
x_0(t) \quad = \quad
\left\{
\aligned
& \frac{t}{2}, \qquad & 0\le t\le \frac{1}{2} \\
& t-\frac{1}{4}, \qquad & \frac{1}{2}\le t\le \frac{3}{4} \\
& 2t-1, \qquad & \frac{3}{4}\le t\le 1
\endaligned
\right.
\end{equation}
and
\begin{equation}\label{x1}
x_1(t) \quad = \quad
\left\{
\aligned
& t, \qquad & 0\le t\le \frac{1}{2} \\
& \frac{t}{2}+\frac{1}{4}, \qquad & \frac{1}{2}\le t \le \frac{3}{4} \\
& t-\frac{1}{8}, \qquad & \frac{3}{4}\le t\le \frac{7}{8} \\
& 2t-1, \qquad & \frac{7}{8}\le t\le 1
\endaligned
\right.
\end{equation}

The 180 degree rotation of the square $[0,1] \times [0,1]$ centered at the point $(\frac{1}{2}, \frac{1}{2})$ induces an order 2 automorphism $\rho$  of the group $F$. It is easy to see that $F$ can be generated by $x_0$ and $x_1$.  This automorphism $\rho$ is given by $\rho(x_0)=x_0^{-1}$ and $\rho(x_1)=x_0x_1x_0^{-2}$.  Using the automorphism $\rho$ of $F$, we form the semi-direct product $G=F\rtimes_{\rho} \mathbb Z_2$. The vector space Hom$(G,\mathbb{R}) = \{ \chi | \chi(x_0) = 0 \} \cong \mathbb{R}^1$, and we use Theorem \ref{main} to show that $\Sigma^1(G)= \{ \pm \infty \}$. To see this, first recall that $\Sigma^1(F)^c = \{ [\chi_1], [\chi_2] \}$ where $\chi_1(x_0) = 1$, $\chi_1(x_1) = 0$, and $\chi_2(x_0) = \chi_2(x_1) = -1$.  Since neither of these points are in $\partial Fix \hat{\rho}$, by Theorem \ref{main}, $\Sigma^1(G) = \Sigma^1(F) \cap \partial Fix \hat{\rho} = \{ \pm \infty \}$.  Since Hom$(G,\mathbb{R})$ is one-dimensional, we have that $\Omega^1(G) = \Sigma^1(G)$.  Using the $\pi/2$-neighborhood result of Theorem \ref{hemi}, it follows that $\Omega^1(F)$ is a single arc which contains the north pole $+\infty$ but not the south pole $-\infty$.  Thus $\Omega^1(F) \cap \partial Fix \hat{\rho} = \{ +\infty \}$.  Thus, $\Omega^1(G) \neq \Omega^1(H) \cap \partial_{\infty}Fix\hat{\nu}$ in general.
\end{ex}

\begin{ex}\label{ex2}
Let $H \cong \langle a,b | b^{-1}ab = a^2 \rangle \times \langle c,d | d^{-1}cd = c^2 \rangle \times \langle x,y \rangle$ (so $H$ is the product of two Baumslag-Solitar groups and the free group on two generators), and define the action $\rho$ of  $\mathbb{Z}_2 \cong \langle t | t^2 = 1 \rangle$ on $H$ by $t \cdot a = c, t \cdot b = d, t \cdot c = a, t \cdot d = b, t \cdot x = y$, and $t \cdot y = x$.  Let $G \cong H \rtimes_{\rho} \mathbb{Z}_2$.  The vector space Hom$(H,\mathbb{R}) \cong \mathbb{R}^4$ as any homomorphism must send $a$ and $c$ to zero, and $Fix\hat{\rho} = \{ \phi | \phi(b)=\phi(d); \phi(x)=\phi(y) \} \cong \mathbb{R}^2$.  The complement of $\Sigma^1(H)$ is the set $\{ [\chi] | \chi(b) = \chi(d) = 0 \} \cup \{ [\chi] | \chi(x) = \chi(y) = \chi(b) = 0; \chi(d) = -1 \} \cup \{ [\chi] | \chi(x) = \chi(y) = \chi(d) = 0; \chi(b) = -1 \}$.  Thus, by Theorem~\ref{index}, the complement of $\Sigma^1(G)$ is the two-point set $\{ [\chi] | \chi(b) = \chi(d) = 0 \}$.  By Theorem~\ref{hemi}, $\Omega^1(G)$ is the two-point set $\{ [\chi] | \chi(x) = \chi(y) = 0 \}$.  However, by Theorem~\ref{product-formula}, $\Omega^1(H) = \{ [\chi] | \chi(x)=\chi(y)=0; \chi(b) > 0; \chi(d) > 0 \}$, so $\Omega^1(H) \cap \partial_{\infty}Fix\hat{\rho}$ is the one-point set $\{ [\chi] | \chi(x)=\chi(y)=0; \chi(b)=\chi(d)=1 \}$.
\end{ex}

We do have the following containments.

\begin{prop}\label{contain}
$\Omega^n(H) \cap \partial_{\infty}Fix\hat{\nu} \subseteq \Omega^n(G) \subseteq \Sigma^n(H) \cap \partial_{\infty} Fix \hat{\nu}$.
\end{prop}

\begin{proof}
To see that $\Omega^n(G) \supseteq \Omega^n(H) \cap \partial_{\infty}Fix\hat{\nu}$, if $[\chi] \in \Omega^n(H) \cap \partial_{\infty}Fix\hat{\nu}$, then the open $\frac{\pi}{2}$-neighborhood of $[\chi]$ in $\partial_{\infty}$Hom$(H,\mathbb{R})$, denoted $N^H_{\pi/2}([\chi])$, is contained in $\Sigma^n(H)$.  Thus, $N^G_{\pi/2}([\chi]) = N^H_{\pi/2}([\chi]) \cap \partial_{\infty} Fix \hat{\nu} \subseteq \Sigma^n(H) \cap \partial_{\infty} Fix \hat{\nu} = \Sigma^n(G)$ which implies $[\chi] \in \Omega^n(G)$.  Further, by Theorem~\ref{index} and remark~\ref{ngeq1}, $\Omega^n(G) \subseteq \Sigma^n(G) = \Sigma^n(H) \cap \partial_{\infty} Fix \hat{\nu}$ which finishes the proof.
\end{proof}

It should be noted that these containments can be strict.  Example~\ref{ex1} shows that the first containment can be strict, and example~\ref{ex2} shows the second containment can be strict. 

Proposition~\ref{contain} leads to the following sufficient conditions to obtain equality.

\begin{thm}\label{finsuff}
$\Omega^n(G) = \Omega^n(H) \cap \partial_{\infty}Fix\widehat{\nu}$ if
\begin{enumerate}
\item Hom$(H,\mathbb{R}) \cong \mathbb{R}^1$,
\item $\Sigma^n(H) = S(H)$ which is the character sphere of Hom$(H,\mathbb{R})$, or
\item $\Sigma^n(H) = \emptyset$.
\end{enumerate}
\end{thm}

\begin{proof}
Each condition implies that $\Sigma^n(H) = \Omega^n(H)$ which gives equality for the left and right ends of the above subset inclusion.
\end{proof}

\begin{rmk}
It is worth noting that many groups $H$ satisfy the conditions in Theorem~\ref{finsuff} such as free groups, free abelian groups, nilpotent groups, polycyclic groups, and the Baumslag-Solitar groups $BS(1,m)$.
\end{rmk}

\begin{ex}\label{ex-n2}
The above conditions are not necessary.  Revisiting example~\ref{ex1}, it was shown in \cite{BGK} that $\Sigma^2(F)$ contains the larger arc but not the smaller arc from $\Sigma^1(F)$.  Therefore, by Theorem~\ref{index} and remark~\ref{ngeq1}, we have $\Sigma^2(F \rtimes_\rho \mathbb{Z}_2) = \Sigma^2(F) \cap \partial_{\infty}Fix\hat{\rho} = \{ + \infty \}$.  Therefore, $\Omega^2(F \rtimes_\rho \mathbb{Z}_2) = \{ + \infty \} = \Omega^2(F) \cap \partial_{\infty}Fix\hat{\rho}$, but Thompson's group $F$ does not satisfy any of the conditions of Theorem~\ref{finsuff}.
\end{ex}

\section{Twisted Conjugacy and the $\Omega$-invariant of extenstions}\label{tcomega}

\subsection{Twisted conjugacy}\label{tc}

Following \cite{TW1}, a group $G$ is said to have the property $R_{\infty}$ if $R(\varphi)=\infty$ for all $\varphi \in {\rm Aut}(G)$ where $R(\varphi)$ denotes the cardinality of the set of $\varphi$-twisted conjugacy classes of elements of $G$ (i.\ e.\ the number of orbits of the left action of $G$ on $G$ via $g \cdot h \mapsto gh\varphi(g)^{-1}$). For instance, $R(1_G)$ is the number of ordinary conjugacy classes of elements of $G$. It has been shown in \cite{KW} that $G$ has property $R_{\infty}$ if $\Omega^n(G)$ consists of a single discrete point. However, for such a group $G$ with $\#\Omega^1(G)=1$, Theorem \ref{hemi} implies that $\Sigma^1(G)\ne - \Sigma^1(G)$ so in particular, $G$ cannot be the fundamental group of a closed $3$-manifold (See \cite[Cor. F]{BNS}). The only known examples of groups $G$ with $\#\Omega^n(G)=1$ are of the form $BS(1,n)\times W$ where $n\ge 2$ and $\#\Omega^n(W) = 0$.

The basic algebraic techniques used in the present paper for showing $R(\varphi)=\infty$ is the relationship among the Reidemeister numbers of group homomorphisms of a short exact sequence. In general, given a commutative diagram of groups and homomorphisms
\begin{equation*}
\begin{CD}
    A    @>{\eta}>>  B  \\
    @V{\psi}VV  @VV{\varphi}V   \\
    A    @>{\eta}>>  B
\end{CD}
\end{equation*}
the homomorphism $\eta$ induces a function $\hat {\eta}:\mathcal R(\psi) \to \mathcal R(\varphi)$ where $\mathcal R(\alpha)$ denotes the set of $\alpha$-twisted conjugacy classes.
For our purposes, we are only concerned with automorphisms.  For more general results, see \cite{daci-peter} and \cite{wong}.  We will use the following lemma; for a proof, see \cite{KW}.

\begin{lem}\label{R-facts}
Consider the following commutative diagram
\begin{equation*}\label{general-Reid}
\begin{CD}
    1 @>>> A    @>>>  B @>>>    C @>>> 1 \\
    @.     @V{\varphi'}VV  @V{\varphi}VV   @V{\overline \varphi}VV @.\\
    1 @>>> A    @>>>  B @>>>    C @>>> 1
 \end{CD}
\end{equation*}
where the rows are short exact sequences of groups and the vertical arrows are group automorphisms.
\begin{enumerate}
\item If $R(\overline \varphi)=\infty$ then $R(\varphi)=\infty$.

\item If $|Fix \overline \varphi|<\infty$ and $R(\varphi')=\infty$ then $R(\varphi)=\infty$.
\end{enumerate}
\end{lem}

\subsection{Using $\Omega$ for finite extensions}\label{finomega}

\begin{thm}\label{fintwist}
Let $G$ be a finite extension given by the short exact sequence of groups
$$
1\to H \stackrel{i}{\to} G \stackrel{p}{\to} K \to 1
$$
where $K$ is finite, $H$ is finitely generated, and let $\nu: K\to G$ be a left transversal for $K$ such that $\nu(1_K)=1_G$.  Let $\varphi \in Aut(G)$ such that $H$ is invariant under $\varphi$.
If $\Omega^n(H) \cap \partial_{\infty}Fix\hat{\nu}$ has exactly one rational point, then $R(\varphi) = \infty$.  In particular, if $H$ is characteristic in $G$ and $\Omega^n(H) \cap \partial_{\infty}Fix\hat{\nu}$ has exactly one rational point, then $G$ has the $R_{\infty}$ property.
\end{thm}

\begin{proof}
Suppose $\varphi \in Aut(G)$ with $\varphi(H) = H$, and suppose $\Omega^n(H) \cap \partial_{\infty}Fix\hat{\nu} = \{[\chi]\}$.  Let $N = ker(\chi)$ and $V := {\rm Hom}(G/N,\mathbb{R})$.  Since $[\chi]$ is rational, $G/N$ has rank $1$, so $V$ is $1$-dimensional.  Define $\tilde{\varphi}:\Hom \to \Hom$ by $\tilde{\varphi}(\alpha) = \alpha \circ \varphi$.  Since $\varphi(H) = H$ and both $\Omega^n(H)$ and $\partial_{\infty}Fix\hat{\nu}$ are invariant under automorphisms, $[\tilde{\varphi}(\chi)] \in \Omega^n(H) \cap \partial_{\infty}Fix\hat{\nu}$, so $[\tilde{\varphi}(\chi)] = [\chi]$.  Thus, $\chi \circ \varphi = c\chi$ for some $c \in \mathbb{Z}$, so $\varphi(N) \subseteq N$ and $N$ is invariant under $\varphi$.

The automorphism $\varphi$ induces the map $\bar{\varphi}:G/N \to G/N$ defined by $\bar{\varphi}(gN) = \varphi(g)N$ and the map $\hat{\varphi}:V \to V$ defined by $\hat{\varphi}(\alpha)(gN) = \alpha(\varphi(g)N)$.  We will show $\bar{\varphi} = id$.  Since $\varphi$ is invertible and $N$ is invariant under $\varphi$, we have that $\hat{\varphi}$ is invertible, and $\{ [ \bar{\chi} ] \}$ is a basis for $V$ where $\bar{\chi}:G/N \to \mathbb{R}$ is induced by $\chi$.  Since $\hat{\varphi}$ is invertible, $c = \pm 1$, but since $-[\chi] \not\in \Omega^n(H) \cap \partial_{\infty}Fix\hat{\nu}$, $c = 1$.  Thus, $\hat{\varphi}(\bar{\chi})=\bar{\chi}$ which implies that $\bar{\chi}(gN) = \bar{\chi}(\varphi(g)N)$, and so $g^{-1}\varphi(g) \in N$.  Therefore, $\varphi(g)N = gN$ which implies $\bar{\varphi}(gN) = \varphi(g)N = gN$.  The free abelian group $G/N$ has rank $1$, so $(G/N)/\{{\rm torsion}\} \cong \mathbb{Z}$, and $\bar{\varphi}$ is the identity on $G/N$, so it is also induces the identity on $(G/N)/\{{\rm torsion}\}$.  It is clear that $R(1_{\mathbb{Z}}) = \infty$.  It follows from Lemma~\ref{R-facts} that $R(\bar{\varphi}) = \infty$, and hence, $R(\varphi) = \infty$.  In particular, if $H$ is characteristic in $G$, then $G$ has property $R_{\infty}$.
\end{proof}

\begin{rmk}\label{finite-set1}
It is known that if $\Omega^n$ of a group is finite, then it contains either $0, 1,$ or $2$ points (in the case with two points, the points are antipodal).  Although $\Omega^n(H) \cap \partial_{\infty}Fix\hat{\nu}$ is not equal to $\Omega^n(G)$ in general, the set $\Omega^n(H) \cap \partial_{\infty}Fix\hat{\nu}$ has the same property that if it is finite, then it contains either $0, 1,$ or $2$ (antipodal) points.  This is due to the fact that if $\Omega^n(H) \cap \partial_{\infty}Fix\hat{\nu}$ contains two non-antipodal points, then by Theorem~\ref{hemi}, the arc joining those points will also be in $\Omega^n(H) \cap \partial_{\infty}Fix\hat{\nu}$.
\end{rmk}

\begin{ex}\label{special}
Revisiting example~\ref{ex2}, we showed that $\Omega^1(G)$ was two antipodal points while $\Omega^1(H) \cap \partial_{\infty}Fix\hat{\rho}$ contains exactly one rational point.  The automorphism $\varphi:G \to G$ defined by $\varphi(a,1) = (a,1)$, $\varphi(b,1) = (b,1)$, $\varphi(c,1) = (c,1)$, $\varphi(d,1) = (d,1)$, $\varphi(1,t) = (1,t)$, $\varphi(x,1) = (y,t)$, and $\varphi(y,1) = (x,t)$ is an order two automorphism of $G$ that is not $H$-invariant, so $H$ is not characteristic in $G$.  However, there are automorphisms of $G$ that are $H$-invariant (for example, send $a \mapsto c$, $b \mapsto d$, $c \mapsto a$, $d \mapsto b$, $x \mapsto y$, $y \mapsto x$, and $t \mapsto t$), and by Theorem~\ref{fintwist}, these automorphisms $\phi$ have $R(\phi)=\infty$.  This information cannot be obtained from either \cite{GK} (since $\Sigma^1(G)^c$ is two antipodal points) or \cite{KW} (since $\Omega^1(G)$ is two antipodal points).
\end{ex}

\begin{ex}\label{omega2}
Revisiting example~\ref{ex-n2}, since $\Omega^2(F) \cap \partial_{\infty}Fix\hat{\rho}$ contains exactly one point, any automorphism $\varphi$ of this group that leaves $F$ invariant would have $R(\varphi) = \infty$.  Note that the map given by $(x_0,1) \mapsto (x_0,1); (x_1,1)\to (x_1,t); (1,t)\mapsto (x_0,t)$ defines an automorphism and it does not preserve $F$ so that $F$ is {\it not} characteristic in $G$. The fact that $F$ is not characteristic in $G$ is the reason that the $R_{\infty}$ property of $G$ does {\it not} follow from case (2) of Lemma~\ref{R-facts} from the fact that $F$ has property $R_{\infty}$ \cite{BFG}.  However, by \cite[Theorem 4.3]{KW}, since $\Omega^2(F \rtimes_\rho \mathbb{Z})$ contains exactly one point, the group $F \rtimes_\rho \mathbb{Z}$ has the $R_{\infty}$ property.  
\end{ex}

\subsection{Using $\Omega$ for split extensions}\label{splomega}

\begin{thm}\label{splittwist}
Let $G$ be a split extension given by the short exact sequence of groups
$$
1\to H \stackrel{i}{\to} G \stackrel{p}{\to} K \to 1
$$
$H$ and $K$ are finitely generated, and let $\nu: K\to G$ be a left transversal for $K$ such that $\nu(1_K)=1_G$.  Let $\varphi \in Aut(G)$ such that $H$ is invariant under $\varphi$.
If $(\Omega^n(H) \cap \partial_{\infty}Fix\hat{\nu}) \circledast \Omega^n(K)$ has exactly one point, then $R(\varphi) = \infty$.  In particular, if $H$ is characteristic in $G$ and $(\Omega^n(H) \cap \partial_{\infty}Fix\hat{\nu}) \circledast \Omega^n(K)$ has exactly one rational point, then $G$ has the $R_{\infty}$ property.
\end{thm}

\begin{proof}
For $(\Omega^n(H) \cap \partial_{\infty}Fix\hat{\nu}) \circledast \Omega^n(K)$ to have exactly one rational point, either $\Omega^n(H) \cap \partial_{\infty}Fix\hat{\nu}$ contains exactly one rational point or $\Omega^n(K)$ contains exactly one rational point.  In the case where $\Omega^n(H) \cap \partial_{\infty}Fix\hat{\nu}$ contains exactly one rational point, the proof follows the proof of Theorem~\ref{fintwist}.

In the case where $\Omega^n(K)$ contains exactly one rational point, it follows from \cite{KW} that $K$ has property $R_{\infty}$.  Since $\varphi$ is $H$-invariant, $\varphi$ induces an automorphism $\bar{\varphi}$ on $K$.  Therefore, $R(\bar{\varphi}) = \infty$.  By Lemma~\ref{R-facts}, we have $R(\varphi) = \infty$.

In particular, if $H$ is characteristic in $G$, then $G$ has property $R_{\infty}$.
\end{proof}

Just as in remark~\ref{finite-set1}, if $(\Omega^1(H) \cap \partial_{\infty}Fix\hat{\nu}) \circledast \Omega^1(K)$ is finite, then it will either contain $0, 1,$ or $2$ (antipodal) points. 

It is fair to wonder about the conjecture that $\Omega^1(G) = (\Omega^1(H) \cap \partial_{\infty}Fix\hat{\nu}) \circledast \Omega^1(K)$.  The authors know of examples where the containment $(\Omega^1(H) \cap \partial_{\infty}Fix\hat{\nu}) \circledast \Omega^1(K) \subseteq \Omega^1(G)$ is false.  Example~\ref{ex1} shows the reverse containment $\Omega^1(G) \subseteq (\Omega^1(H) \cap \partial_{\infty}Fix\hat{\nu}) \circledast \Omega^1(K)$ is also false.  Are there sufficient and/or necessary conditions where the conjecture does hold?

\noindent
{\Small
\begin{tabbing}
Nic Koban \hphantom{xxxxxxxxxxxxxxxxxxxxxxxxxx} \= Peter Wong\\
Department of Mathematics \> Department of Mathematics\\
University of Maine Farmington \> Bates College\\
Farmington, ME 04938 \> Lewiston, ME 04240 \\
USA \> USA\\
nicholas.koban@maine.edu \> pwong@bates.edu\\
\end{tabbing}
}
\end{document}